\documentclass{amsart}
\usepackage{amsfonts}

\newtheorem{theorem}{Theorem}
\theoremstyle{plain}

\newtheorem{corollary}{Corollary}

\newtheorem{definition}{Definition}
\newtheorem{example}{Example}

\newtheorem{proposition}{Proposition}
\newtheorem{remark}{Remark}

\numberwithin{equation}{section}
\input{tcilatex}

\begin{document}
\title[ Fibonacci sequence via $\ $the $\sum -$ trnasform]{The Fibonacci
sequence via the $\sum -$ transform}
\author{Dejenie A. Lakew}
\address{John Tyler Community College\\
Department of Mathematics}
\email{dlakew@jtcc.edu}
\urladdr{http://www.jtcc.edu}
\date{January 1, 2014}

\begin{abstract}
In this short article, we study different problems described as initial
value problems of discrete differential equations and develop a \ a
transform method called the $\sum -$ transform, a discrete version of the
continuous Laplace transform to generate solutions as rational functions of
integers to these initial value problems. Particularly we look how the
method generates the traditionally known numbers called Fibonacci sequence
as a solution to an initial value problem of a discrete differential
equation.
\end{abstract}

\maketitle

\section{\protect\bigskip Introduction}

In this short article we introduce a discrete analogue of the continuous
Laplace transform, called the $\sum -$ transform or discrete Laplace
transform (DLT) to study solutions to some initial value problems of
discrete difference equations. The solutions will be sequences of numbers
not necessarily integers but rational functions of integer polynomials. The
Fibonacci numbers we know will be particular case of the general sequence we
obtain through such process.

\ 

The Fibonacci numbers are one of the wonders of old mathematics. They
represent several natural things, such as leaves of trees, foliages of
flowers, replications of some species, etc. These Fibonacci numbers are
given by the sequence : 
\begin{equation*}
1,\text{ }1,\text{ }2,\text{ }3,\text{ }5,\text{ }8,\text{ }13,\text{ }21,%
\text{ }34,.....
\end{equation*}%
The pattern is that a number is the sum of two of the previous or
predecessor numbers and can be written as :

\begin{equation}
\left\{ 
\begin{array}{c}
a_{n+2}=a_{n+1}+a_{n}\text{ \ } \\ 
\text{\ \ } \\ 
a_{1}=1,\text{ }a_{2}=1\text{ \ \ for }n=1,\text{ }2,\text{ }3,....%
\end{array}%
\right.  \label{Fibeqn}
\end{equation}

\ 

We can rewrite equation \ref{Fibeqn} as 
\begin{equation*}
\left\{ 
\begin{array}{c}
\Delta a_{n+1}=a_{n}\text{\ } \\ 
a_{1}=1,a_{2}=1%
\end{array}%
\right.
\end{equation*}

We know that the Fibonacci sequence has a formula that generates the
numbers. That formula is :

\ \ 
\begin{equation}
a_{n}=\frac{\left( 1+\sqrt{5}\right) ^{n}-\left( 1-\sqrt{5}\right) ^{n}}{%
2^{n}\sqrt{5}}  \label{Fib}
\end{equation}

\ \ \ \ \ \ \ \ \ \ \ \ \ \ \ \ \ \ \ 

In discrete mathematics or sequence courses we ask students to verify using
mathematical induction that indeed, the given formula generates the
Fibonacci numbers. It is also a known fact that the quotients of consecutive
numbers of the sequence converges to a number: 
\begin{equation*}
\frac{a_{n+1}}{a_{n}}\longrightarrow \frac{1+\sqrt{5}}{2}\text{ as }%
n\rightarrow \infty
\end{equation*}

\ 

We use a powerful method, the $\sum -$ transform or the \textit{discrete
version of the Laplace transform} (for more reading on the transform, see 
\cite{DL}) that generates solutions to many sequential or discrete initial
value problems of difference equations which are prevalent in discrete
mathematics and some applicable mathematics. We will also see how the method
is used to find solutions of a different kind of discrete differential
equations such as:

\begin{equation*}
\left\{ 
\begin{array}{c}
a_{n+1}=\lambda a_{n}+\beta \\ 
a_{1}=a(1)\text{, }n=1,\text{ }2,\text{ }3,..%
\end{array}%
\right.
\end{equation*}

\ 

and 
\begin{equation*}
\left\{ 
\begin{array}{c}
a_{n+2}=a_{n+1}+a_{n} \\ 
a_{1}=a\left( 1\right) ,\text{ }a_{2}=a\left( 2\right) \text{, }n=1,\text{ }%
2,\text{ }3,...%
\end{array}%
\right.
\end{equation*}

\ \ 

\ \ \ \ \ \ \ \ \ \ \ \ \ \ \ \ \ \ \ \ \ \ \ \ \ \ \ \ 

\section{\protect\bigskip The $\sum -$Transform}

\begin{definition}
Let 
\begin{equation*}
f:%
\mathbb{N}
\rightarrow 
\mathbb{R}%
\end{equation*}%
be a sequence and let $s>0$. \ We define the $\sum -$transform or the
discrete Laplace transform of $f$ \ by%
\begin{equation*}
\ell _{d}\left\{ f\left( n\right) \right\} \left( s\right)
:=\dsum\limits_{n=1}^{\infty }f\left( n\right) e^{-sn}
\end{equation*}%
provided the series converges.
\end{definition}

\ \ \ \ \ 

\begin{theorem}
Existence of a $\sum -$ transform\textbf{\ }

Let $f:%
\mathbb{N}
\rightarrow 
\mathbb{R}
$ be a sequence such that 
\begin{equation*}
\mid f\left( n\right) \mid \leq \alpha e^{s_{0}n}\text{ for }\alpha >0,\text{
}s_{0}>0
\end{equation*}%
Then 
\begin{equation*}
\dsum\limits_{n=1}^{\infty }f\left( n\right) e^{-sn}
\end{equation*}

is absolutely convergent and hence is convergent. Therefore, for such a
sequence, the discrete Laplace transform

\begin{equation*}
\ell _{d}\left\{ f\left( n\right) \right\} \left( s\right)
\end{equation*}%
exists finitely for $s>s_{0}$.
\end{theorem}

\begin{proof}
Since 
\begin{equation*}
\mid \dsum\limits_{n=1}^{\infty }f\left( n\right) e^{-sn}\mid \leq
\dsum\limits_{n=1}^{\infty }\alpha e^{\left( s_{0}-s\right) n}=\frac{\alpha 
}{e^{s-s_{0}}-1}<+\infty
\end{equation*}

\ \ 

for $s>s_{0}$, we conclude that sequences which are polynomials in $n$ have
convergent $\sum -$ transforms or discrete Laplace transform.
\end{proof}

\ 

\begin{proposition}
(Transform of translate of a sequence). For $k\in 
\mathbb{N}
$,

\begin{equation*}
\ell _{d}\left\{ f\left( n+k\right) \right\} \left( s\right) =e^{ks}\ell
_{d}\left\{ f\left( n\right) \right\} \left( s\right)
-\dsum\limits_{i=1}^{k}f\left( i\right) e^{\left( k-i\right) s}
\end{equation*}
\end{proposition}

\begin{proof}
Let $f:%
\mathbb{N}
\rightarrow 
\mathbb{R}
$ be a sequence. Then

\begin{eqnarray*}
\ell _{d}\left\{ f\left( n+k\right) \right\} \left( s\right)
&=&\dsum\limits_{n=1}^{\infty }f\left( n+k\right) e^{-sn} \\
&=&\dsum\limits_{m=k+1}^{\infty }f\left( m\right) e^{-s\left( m-k\right) } \\
&=&e^{sk}\dsum\limits_{m=k+1}^{\infty }f\left( m\right) e^{-sm} \\
&=&e^{ks}\dsum\limits_{m=1}^{\infty }f\left( m\right)
e^{-sm}-\dsum\limits_{i=1}^{k}f\left( i\right) e^{\left( k-i\right) s} \\
&=&e^{ks}\ell _{d}\left\{ f\left( n\right) \right\} \left( s\right)
-\dsum\limits_{i=1}^{k}f\left( i\right) e^{\left( k-i\right) s}
\end{eqnarray*}
\end{proof}

\begin{corollary}
\begin{equation*}
\ell _{d}\left\{ f\left( n+1\right) \right\} \left( s\right) =e^{s}\ell
_{d}\left\{ f\left( n\right) \right\} \left( s\right) -f\left( 1\right)
\end{equation*}
\end{corollary}

\begin{proposition}
For $0<a<e^{s}$ we have%
\begin{equation*}
\ell _{d}\left\{ a^{n-1}\right\} \left( s\right) =\frac{1}{e^{s}-a}
\end{equation*}
\end{proposition}

\begin{proof}
From the definition,

\begin{eqnarray*}
\ell _{d}\left\{ a^{n-1}\right\} \left( s\right)
&=&\dsum\limits_{n=1}^{\infty }e^{-sn}a^{n-1} \\
&=&\dsum\limits_{n=1}^{\infty }e^{-sn}e^{\left( n-1\right) \ln a} \\
&=&a^{-1}\dsum\limits_{n=1}^{\infty }e^{-sn}e^{n\ln a} \\
&=&a^{-1}\dsum\limits_{n=1}^{\infty }e^{-\left( s-\ln a\right) n} \\
&=&\frac{1}{e^{s}-a}
\end{eqnarray*}
\end{proof}

\begin{example}
For $a=5$,%
\begin{equation*}
\ell _{d}\left\{ 5^{n-1}\right\} \left( s\right) =\frac{1}{e^{s}-5}
\end{equation*}
\end{example}

\begin{definition}
Let $f:%
\mathbb{N}
\rightarrow 
\mathbb{R}
$ be a sequence. The discrete derivative of of $f$ denoted%
\begin{equation*}
\triangle f\left( n\right) :=f\left( n+1\right) -f\left( n\right)
\end{equation*}
\end{definition}

\begin{proposition}
(Transform of a discrete derivative of a sequence).

\begin{equation*}
\ell _{d}\left\{ \triangle f\left( n\right) \right\} \left( s\right) =\left(
e^{s}-1\right) \ell _{d}\left\{ f\left( n\right) \right\} -f\left( 1\right)
\end{equation*}
\end{proposition}

\begin{proof}
\begin{equation*}
\ell _{d}\left\{ \triangle f\left( n\right) \right\} \left( s\right)
=\dsum\limits_{n=1}^{\infty }\triangle f\left( n\right)
e^{-sn}=\dsum\limits_{n=1}^{\infty }\left( f\left( n+1\right) -f(n)\right)
e^{-sn}
\end{equation*}

\begin{eqnarray*}
&=&\dsum\limits_{n=1}^{\infty }f\left( n+1\right)
e^{-sn}-\dsum\limits_{n=1}^{\infty }f\left( n\right) e^{-sn} \\
&=&\ell _{d}\left\{ f\left( n+1\right) \right\} -\ell _{d}\left\{ f\left(
n\right) \right\} \\
&=&\left( e^{s}-1\right) \ell _{d}\left\{ f\left( n\right) \right\} \left(
s\right) -f\left( 1\right)
\end{eqnarray*}
\end{proof}

\bigskip\ \ 

Next we define a discrete convolution operator on sequences which latter
will be useful in solving discrete initial value problems.

\ 

\begin{definition}
Let $f,$ $g:%
\mathbb{N}
\rightarrow 
\mathbb{R}
$ be two sequences . Then the discrete convolution of $f$ and $g$ denoted $%
\left( f\ast g\right) \left( n\right) $ is defined by

\begin{equation*}
\left( f\ast g\right) \left( n\right) :=\dsum\limits_{k=1}^{n-1}f\left(
k\right) g\left( n-k\right)
\end{equation*}
\end{definition}

\begin{example}
$\left( n\ast 1\right) =\frac{n^{2}-n}{2}$
\end{example}

\begin{example}
$\left( n\ast n\right) =\frac{n^{3}-n}{6}$
\end{example}

\begin{proposition}
(Transform of a discrete convolution).

\begin{equation*}
\ell _{d}\left\{ \left( f\ast g\right) \left( n\right) \right\} \left(
s\right) =\ell _{d}\left\{ f\left( n\right) \right\} \ell _{d}\left\{
g\left( n\right) \right\}
\end{equation*}
\end{proposition}

\begin{proof}
From the product of the two series: 
\begin{equation*}
\left( \dsum\limits_{n=1}^{\infty }a_{n}x^{n}\right) \left(
\dsum\limits_{n=1}^{\infty }b_{n}x^{n}\right) =\dsum\limits_{n=2}^{\infty
}c_{n}x^{n}
\end{equation*}

where $c_{n}=\dsum\limits_{k=1}^{n-1}a_{k}b_{n-k}$, we have,

\begin{eqnarray*}
\ell _{d}\left\{ \left( f\ast g\right) \left( n\right) \right\} \left(
s\right) &=&\dsum\limits_{n=1}^{\infty }\left( f\ast g\right) \left(
n\right) e^{-sn} \\
&=&\dsum\limits_{n=2}^{\infty }\left( \dsum\limits_{k=1}^{n-1}f\left(
k\right) g\left( n-k\right) \right) e^{-sn} \\
&=&\left( \dsum\limits_{n=1}^{\infty }f\left( n\right) e^{-sn}\right) \left(
\dsum\limits_{n=1}^{\infty }g\left( n\right) e^{-sn}\right) \\
&=&\ell _{d}\left\{ f\left( n\right) \right\} \ell _{d}\left\{ g\left(
n\right) \right\}
\end{eqnarray*}
\end{proof}

\ 

\begin{corollary}
\begin{equation*}
\ell _{d}\left\{ \dsum\limits_{k=1}^{n-1}f\left( k\right) \right\} \left(
s\right) =\frac{s\ell _{d}\left\{ f\left( n\right) \right\} }{e^{s}-1}
\end{equation*}
\end{corollary}

\ 

\begin{proof}
Follows from the fact that choosing $g\equiv 1$, we have

\begin{equation*}
\left( f\ast g\right) \left( n\right) =f\left( n\right) \ast
1=\dsum\limits_{k=1}^{n-1}f\left( k\right)
\end{equation*}%
Then taking the transform of both sides, we have the result.
\end{proof}

\begin{proposition}
For a sequence $f:%
\mathbb{N}
\rightarrow 
\mathbb{R}
$,

\begin{equation*}
\ell _{d}\left\{ nf\left( n\right) \right\} \left( s\right) =-\frac{d}{ds}%
\left( \ell _{d}\left\{ f\left( n\right) \right\} \left( s\right) \right)
\end{equation*}
\end{proposition}

\begin{proof}
\begin{eqnarray*}
\frac{d}{ds}\left( \ell _{d}\left\{ f\left( n\right) \right\} \left(
s\right) \right) &=&\frac{d}{ds}\dsum\limits_{n=1}^{\infty }f\left( n\right)
e^{-sn} \\
&=&\dsum\limits_{n=1}^{\infty }\left( -nf\left( n\right) e^{-sn}\right) \\
&=&-\ell _{d}\left\{ nf\left( n\right) \right\} \left( s\right)
\end{eqnarray*}
\end{proof}

\begin{corollary}
For $k\in 
\mathbb{N}
$, 
\begin{equation*}
\ell _{d}\left\{ n^{k}f\left( n\right) \right\} \left( s\right) =\left(
-1\right) ^{k}\frac{d^{k}}{ds^{k}}\ell _{d}\left\{ f\left( n\right) \right\}
\left( s\right)
\end{equation*}
\end{corollary}

\begin{remark}
By taking $f\equiv 1$, we get the relation:

\begin{eqnarray*}
\ \ell _{d}\left\{ n^{k}\right\} \left( s\right) &=&\left( -1\right) ^{k}%
\frac{d^{k}}{ds^{k}}\ell _{d}\left\{ 1\right\} \left( s\right) \\
&=&\left( -1\right) ^{k}\frac{d^{k}}{ds^{k}}\left( \frac{1}{e^{s}-1}\right)
\end{eqnarray*}
\end{remark}

\ \ \ \ \ \ \ \ \ \ \ \ \ \ \ \ \ \ \ \ \ \ \ \ \ \ \ \ \ \ \ \ \ \ \ \ \ \
\ \ \ \ \ \ \ \ 

\section{Ivps of discrete differential equations.}

In this section we solve initial value problems of discrete differential
equations using the $\sum -$ transform and the Fibonacci numbers

\begin{proposition}
(\cite{DL}) 
\begin{equation*}
\ell _{d}\left\{ \frac{1}{n}\right\} \left( s\right) =s-\ln \left(
e^{s}-1\right)
\end{equation*}%
\ for $s>0$.
\end{proposition}

\begin{proposition}
For $n\geq 2$, the IVP :%
\begin{equation*}
\left\{ 
\begin{array}{c}
\triangle f\left( n\right) =\frac{1}{n^{2}} \\ 
f\left( 2\right) =2%
\end{array}%
\right.
\end{equation*}

has solution given by%
\begin{equation*}
f\left( n\right) =1+\dsum\limits_{k=1}^{n-1}\frac{1}{k^{2}}
\end{equation*}
\end{proposition}

\begin{proof}
Re-writing the difference equation as : $n\triangle f\left( n\right) =\frac{1%
}{n}$, taking the transform of both sides and using corollary $3.3$ we get

\begin{equation*}
\frac{d}{ds}\ell _{d}\left\{ f\left( n\right) \right\} \left( s\right) +%
\frac{e^{s}}{e^{s}-1}\ell _{d}\left\{ f\left( n\right) \right\} \left(
s\right) =-\frac{s-\ln \left( e^{s}-1\right) }{e^{s}-1}.
\end{equation*}%
Again solving for $\ell _{d}\left\{ f\left( n\right) \right\} \left(
s\right) $, we have

\begin{equation*}
\ \ell _{d}\left\{ f\left( n\right) \right\} \left( s\right) =\frac{1}{%
e^{s}-1}-\frac{1}{e^{s}-1}\int \left( s-\ln \left( e^{s}-1\right) \right) ds.
\end{equation*}

\begin{equation*}
\Rightarrow \text{ }f\left( n\right) =1-\ell _{d}^{-1}\left\{ \frac{1}{%
e^{s}-1}\right\} \ast \ell _{d}^{-1}\left\{ \int \left( s-\ln \left(
e^{s}-1\right) \right) ds\right\}
\end{equation*}

\begin{eqnarray*}
&=&1-\left( 1\ast \left( -\frac{1}{n^{2}}\right) \right) \\
&=&1+\dsum\limits_{k=1}^{n-1}\frac{1}{k^{2}}.
\end{eqnarray*}
\end{proof}

\ \ \ 

\begin{proposition}
The second order IVP:%
\begin{equation*}
\left\{ 
\begin{array}{c}
\triangle ^{2}f\left( n\right) =n \\ 
\text{ \ }f\left( 1\right) =1,\text{ \ }\triangle f\left( 1\right) =2%
\end{array}%
\right.
\end{equation*}

has solution given by :%
\begin{equation*}
f\left( n\right) =2n-1+\frac{n\left( n-1\right) \left( n-2\right) }{6}
\end{equation*}
\end{proposition}

\begin{proof}
First,%
\begin{eqnarray*}
\triangle ^{2}f\left( n\right) &=&\triangle \left( \triangle f\left(
n\right) \right) \\
&=&f\left( n+2\right) -2f\left( n+1\right) +f\left( n\right)
\end{eqnarray*}
and using the initial conditions we get:

\begin{equation*}
\ell _{d}\left\{ \triangle ^{2}f\left( n\right) \right\} \left( s\right)
=\left( e^{2s}-2e^{s}+1\right) \ell _{d}\left\{ f\left( n\right) \right\}
\left( s\right) -e^{s}-1.
\end{equation*}

\begin{equation*}
\Rightarrow \text{ \ \ \ \ \ \ \ }\left( e^{s}-1\right) ^{2}\ell _{d}\left\{
f\left( n\right) \right\} \left( s\right) -e^{s}-1=\frac{e^{s}}{\left(
e^{s}-1\right) ^{2}}
\end{equation*}

\begin{equation*}
\Rightarrow \text{ \ \ \ \ \ \ \ }\ell _{d}\left\{ f\left( n\right) \right\}
\left( s\right) =\frac{1}{\left( e^{s}-1\right) ^{2}}+\frac{e^{s}}{\left(
e^{s}-1\right) ^{2}}+\frac{e^{s}}{\left( e^{s}-1\right) ^{4}}
\end{equation*}

\ 

Then taking the inverse transform and using convolutions, we get the
solution as :

\begin{eqnarray*}
f\left( n\right) &=&\left( 1\ast 1\right) +n+\frac{n\left( n-1\right) \left(
n-2\right) }{6} \\
&=&2n-1+\frac{n\left( n-1\right) \left( n-2\right) }{6}
\end{eqnarray*}
\end{proof}

\ \ 

\section{ The Fibonacci numbers via the $\sum -$transform.}

\begin{proposition}
(The main result ) The sequence of numbers: 
\begin{equation*}
1,\text{ }1,\text{ }2,\text{ }3,\text{ }5,\text{ }8,\text{ }13,\text{ }....
\end{equation*}%
usually called the Fibonacci sequence are generated by the formula : 
\begin{equation*}
a_{n}=\frac{\left( 1+\sqrt{5}\right) ^{n}-\left( 1-\sqrt{5}\right) ^{n}}{%
2^{n}\sqrt{5}}
\end{equation*}

where $n\in 
\mathbb{N}
$.
\end{proposition}

\ 

\begin{proof}
First, we all know, the numbers can be represented in a recursive way by :

\begin{equation*}
\left\{ 
\begin{array}{c}
a_{n+2}=a_{n+1}+a_{n}\text{ \ for }n=1,2,3,.... \\ 
\text{\ }a_{1}=1,\text{ }a_{2}=1.%
\end{array}%
\right.
\end{equation*}

\ 

In tern the above recursive definition can also be described as an initial
value problem of a second order discrete difference equation given by :

\ 
\begin{equation*}
\left\{ 
\begin{array}{c}
\Delta ^{2}a_{n}+\Delta a_{n}=a_{n} \\ 
a_{1}=1,\text{ }a_{2}=1%
\end{array}%
\right.
\end{equation*}

\ 

We will see in a moment that this sequence as a special case of a general
sequence that will be generated from the one whose formula will be obtained
from the recursive expression :

\ 
\begin{equation*}
\left\{ 
\begin{array}{c}
a_{n+2}=a_{n+1}+a_{n}\text{ \ for }n=1,2,3,.... \\ 
\text{\ }a_{1}=a(1),\text{ }a_{2}=a(2)%
\end{array}%
\right.
\end{equation*}

\ 

as two initial conditions. The pattern that will be observed from the latter
sequence is that the general term of the sequence will appear as a linear
combination of the two initial conditions $a_{1}$ and $a_{2}$:

\ 
\begin{equation*}
a_{n}=\gamma _{n}a_{1}+\beta _{n}a_{2}
\end{equation*}

\ \ 

in which $\gamma _{n},$ $\beta _{n}$ are themselves obtained as Fibonacci
numbers of the respective coefficients of $a_{1}$ and $a_{2}$. I call these
numbers Fibonacci - like numbers.

\ 

Applying the discrete Laplace transform on both sides of the later recursive
equation :

\begin{eqnarray*}
\ell _{d}\{a_{n+2}\}(s) &=&\ell _{d}\{a_{n+1}+a_{n}\}(s) \\
&=&\ell _{d}\{a_{n+1}\}(s)+\ell _{d}\{a_{n}\}(s).
\end{eqnarray*}

\ 

From results of (\cite{DL}) we have transforms of these types:

\begin{eqnarray*}
\ell _{d}\{a_{n+2}\}(s) &=&e^{2s}\ell _{d}\{a_{n}\}(s)-e^{s}a_{1}-a_{2} \\
\ell _{d}\{a_{n+1}\}(s)+\ell _{d}\{a_{n}\}(s) &=&e^{s}\ell
_{d}\{a_{n}\}(s)-a_{1}+\ell _{d}\{a_{n}\}(s)
\end{eqnarray*}

$\Rightarrow $%
\begin{equation*}
e^{2s}\ell _{d}\{a_{n}\}(s)-e^{s}a_{1}-a_{2}=e^{s}\ell
_{d}\{a_{n}\}(s)-a_{1}+\ell _{d}\{a_{n}\}(s)
\end{equation*}

Therefore rearranging, we have : 
\begin{equation*}
\ell _{d}\{a_{n}\}(s)=\frac{a_{1}e^{s}+a_{2}-a_{1}}{e^{2s}-e^{s}-1}
\end{equation*}

But 
\begin{equation*}
e^{2s}-e^{s}-1=\left( e^{s}-\frac{\left( 1+\sqrt{5}\right) }{2}\right)
\left( e^{s}-\frac{\left( 1-\sqrt{5}\right) }{2}\right)
\end{equation*}%
and

\begin{eqnarray*}
\frac{a_{1}e^{s}+a_{2}-a_{1}}{e^{2s}-e^{s}-1} &=&\frac{a_{1}e^{s}+a_{2}-a_{1}%
}{\left( e^{s}-\frac{\left( 1+\sqrt{5}\right) }{2}\right) \left( e^{s}-\frac{%
\left( 1-\sqrt{5}\right) }{2}\right) } \\
&=&\frac{a_{1}\left( \sqrt{5}-1\right) +2a_{2}}{2\sqrt{5}\left( e^{s}-\frac{%
\left( 1+\sqrt{5}\right) }{2}\right) }+\frac{a_{1}\left( \sqrt{5}+1\right)
-2a_{2}}{2\sqrt{5}\left( e^{s}-\frac{\left( 1-\sqrt{5}\right) }{2}\right) }
\end{eqnarray*}

\ 

Therefore taking the inverse discrete Laplace transform we have :

\begin{equation*}
a_{n}=\ell _{d}^{-1}\left\{ \frac{a_{1}\left( \sqrt{5}-1\right) +2a_{2}}{2%
\sqrt{5}\left( e^{s}-\frac{\left( 1+\sqrt{5}\right) }{2}\right) }+\frac{%
a_{1}\left( \sqrt{5}+1\right) -2a_{2}}{2\sqrt{5}\left( e^{s}-\frac{\left( 1-%
\sqrt{5}\right) }{2}\right) }\right\}
\end{equation*}

\begin{equation*}
\ \ \ =\frac{1}{2^{n}\sqrt{5}}\left( 
\begin{array}{c}
\left( a_{1}\left( \sqrt{5}-1\right) +2a_{2}\right) \left( 1+\sqrt{5}\right)
^{n-1} \\ 
+\left( a_{1}\left( \sqrt{5}+1\right) -2a_{2}\right) \left( 1-\sqrt{5}%
\right) ^{n-1}%
\end{array}%
\right)
\end{equation*}

\ 

Therefore, the sequence

\begin{equation*}
a_{n}=\frac{1}{2^{n}\sqrt{5}}\left( 
\begin{array}{c}
\left( a_{1}\left( \sqrt{5}-1\right) +2a_{2}\right) \left( 1+\sqrt{5}\right)
^{n-1} \\ 
+\left( a_{1}\left( \sqrt{5}+1\right) -2a_{2}\right) \left( 1-\sqrt{5}%
\right) ^{n-1}%
\end{array}%
\right)
\end{equation*}

\ 

is a solution to the recursive equation of the Fibonacci-like numbers.
Rearranging the expression, we get linear combinations of $a_{1}$ and $a_{2}$
as :

\begin{eqnarray*}
a_{n} &=&\frac{\left( \left( \sqrt{5}-1\right) \left( 1+\sqrt{5}\right)
^{n-1}+\left( \sqrt{5}+1\right) ^{n-1}-\left( 1-\sqrt{5}\right)
^{n-1}\right) }{2^{n}\sqrt{5}}a_{1} \\
&&+\frac{\left( \left( 1+\sqrt{5}\right) ^{n-1}-\left( 1-\sqrt{5}\right)
^{n-1}\right) }{2^{n-1}\sqrt{5}}a_{2}
\end{eqnarray*}

\ 

with coefficients of $a_{1}$ and $a_{2}$ being represented by:

\ 
\begin{equation*}
\gamma _{n}=\frac{\left( \left( \sqrt{5}-1\right) \left( 1+\sqrt{5}\right)
^{n-1}+\left( \sqrt{5}+1\right) ^{n-1}-\left( 1-\sqrt{5}\right)
^{n-1}\right) }{2^{n}\sqrt{5}}
\end{equation*}%
and 
\begin{equation*}
\beta _{n}=\frac{\left( \left( 1+\sqrt{5}\right) ^{n-1}-\left( 1-\sqrt{5}%
\right) ^{n-1}\right) }{2^{n-1}\sqrt{5}}
\end{equation*}

\ \ \ \ 

We note that when $n=1$, the coefficient of $a_{1}$ is $1$ and that of $%
a_{2} $ is zero and therefore the term will be just $a_{1}$. Like wise when $%
n=2$, the coefficient of $a_{1}$ is zero and that of $a_{2}$ is one and
again the term will be $a_{2}.$ The coefficients themselves are generated as
a sequence which are Fibonacci like numbers.

\ 

Then coming back to our original question, extracting a sequence that
generates the Fibonacci sequence, we look at the above general sequence but
with two fixed initial conditions :

\ 
\begin{equation*}
a_{1}=1=a_{2}
\end{equation*}

\ 

These initial conditions provide the following sequence which generates the
well known Fibonacci numbers (\ref{Fib})\ that are known to be integers:

\begin{equation*}
a_{n}=\frac{\left( 1+\sqrt{5}\right) ^{n}-\left( 1-\sqrt{5}\right) ^{n}}{%
2^{n}\sqrt{5}},\text{ }n\in 
\mathbb{N}%
\end{equation*}

\ 
\end{proof}

\ 

The other sequences obtained with initial conditions other than $1$ will
generate numbers which are not necessarily integers but ruled by the
sequential definition indicated at the beginning.\newline
Therefore we have sequences that are Fibonacci like but non integer valued.
It will be a challenge to find which initial conditions will provide a
sequence of integer values other than the Fibonacci sequence. It is a
problem that somebody can pursue.

\ 

\begin{proposition}
Let $\lambda \left( \neq 1\right) \in 
\mathbb{R}
$, the solution to the recursive equation: 
\begin{equation*}
\left\{ 
\begin{array}{c}
a_{n+1}=\lambda a_{n}+\beta \\ 
a\left( 1\right) =a_{1},\text{ }n\in 
\mathbb{N}%
\end{array}%
\right.
\end{equation*}

is given by 
\begin{equation*}
a_{n}=\left( a_{1}+\frac{\beta }{\lambda -1}\right) \lambda ^{n-1}+\frac{%
\beta }{1-\lambda },\text{ }n\in 
\mathbb{N}%
\end{equation*}
\end{proposition}

\ \ 

\begin{proof}
Using the discrete replace transform of both sides of the equation in the
proposition, we have:

\begin{eqnarray*}
l_{d}\{a_{n}\} &=&\frac{a_{1}}{e^{s}-\lambda }+\frac{\beta }{\left(
e^{s}-1\right) \left( e^{s}-\lambda \right) } \\
&=&\frac{a_{1}}{e^{s}-\lambda }+\frac{\beta }{\left( \lambda -1\right)
\left( e^{s}-1\right) }+\frac{\beta }{\left( 1-\lambda \right) \left(
e^{s}-1\right) } \\
&=&\frac{\left( a_{1}+\frac{\beta }{\lambda -1}\right) }{e^{s}-\lambda }+%
\frac{\beta }{\left( 1-\lambda \right) \left( e^{s}-1\right) }
\end{eqnarray*}%
Then taking the inverse discrete Laplace transform, we have: 
\begin{equation*}
a_{n}=\left( a_{1}+\frac{\beta }{\lambda -1}\right) \lambda ^{n-1}+\frac{%
\beta }{1-\lambda },\text{ }n\in 
\mathbb{N}%
\end{equation*}

\ 

Note here why we restrict $\lambda \neq 1$.

\ 

The case for $\lambda =1$ is done in the following way: considering the
equation : 
\begin{equation*}
a_{n+1}=a_{n}+\beta ,\text{ \ }n=1,\text{ }2,\text{ }3,...
\end{equation*}

\ 

and taking the discrete Laplace transform of both sides, and solving for $%
l_{d}\{a_{n}\}$ we get

\begin{equation*}
l_{d}\{a_{n}\}=\frac{a_{1}}{e^{s}-1}+\frac{\beta }{\left( e^{s}-1\right) ^{2}%
}
\end{equation*}%
Applying the inverse discrete Laplace transform we have, $\ $%
\begin{eqnarray*}
a_{n} &=&l_{d}^{-1}\left\{ \frac{a_{1}}{e^{s}-1}\right\} +l_{d}^{-1}\left\{ 
\frac{\beta }{\left( e^{s}-1\right) ^{2}}\right\} \\
&=&a_{1}l_{d}^{-1}\left\{ \frac{1}{e^{s}-1}\right\} +\beta l_{d}^{-1}\left\{ 
\frac{1}{\left( e^{s}-1\right) ^{2}}\right\} \\
&=&a_{1}+\beta \left( 1\ast 1\right) \\
&=&a_{1}+\beta \left( n-1\right)
\end{eqnarray*}

\ \ \ \ \ \ \ \ \ \ \ \ 

where $1\ast 1$ is the convolution of the constant sequence $1$ by itself,
which is $n-1$. Therefore this case has a solution given by :

\ 
\begin{equation*}
a_{n}=\beta \left( n-1\right) +a_{1},\text{ \ }n\in 
\mathbb{N}%
\end{equation*}
\end{proof}

\section{$\protect\bigskip $}

\section{The $\infty -$order initial value problem (IVP$_{\infty }$)}

In this section we investigate the infinite order differential operator $%
\sum_{k=0}^{\infty }\frac{D^{k}}{k!}$ with special infinite number of
initial conditions, where $D=\frac{d}{dx}$. An infinite order initial value
problem for this differential operator can be stated as follow: 
\begin{equation*}
IVP_{\infty }:\left\{ 
\begin{array}{c}
\sum_{k=0}^{\infty }\frac{D^{k}}{k!}f(x)=g(x) \\ 
f^{\left( j\right) }\left( x_{0}\right) =y_{0,j},\text{ \ }j=0,\text{ }1,%
\text{ }2,...%
\end{array}%
\right.
\end{equation*}

\begin{proposition}
The infinite order IVP$_{\infty }$:%
\begin{equation*}
\left\{ 
\begin{array}{c}
\sum_{k=0}^{\infty }\frac{D^{k}}{k!}f(x)=\cos (x),\text{ \ for }x\in \lbrack
0,\infty ) \\ 
f^{\left( k\right) }\left( 0\right) =0,\text{ }\forall k\in 
\mathbb{N}
\cup \{0\}%
\end{array}%
\right.
\end{equation*}%
has a solution given by%
\begin{equation*}
f(x)=\cos \left( x-1\right) u(x-1)
\end{equation*}
\end{proposition}

\begin{proof}
Using the Laplace transform :%
\begin{eqnarray*}
\int_{0}^{\infty }e^{-sx}\sum_{k=0}^{\infty }\frac{D^{k}}{k!}f(x)dx
&=&\int_{0}^{\infty }e^{-sx}\cos \left( x\right) dx \\
&=&\frac{s}{1+s^{2}}
\end{eqnarray*}%
But the left side of the equation becomes : 
\begin{eqnarray*}
\sum_{k=0}^{\infty }\frac{s^{k}}{k!}\int_{0}^{\infty }e^{-sx}f(x)dx
&=&F(s)\sum_{k=0}^{\infty }\frac{s^{k}}{k!} \\
&=&e^{s}F(s)
\end{eqnarray*}

Thus 
\begin{equation*}
e^{s}F(s)=\frac{s}{1+s^{2}}
\end{equation*}
\ which implies%
\begin{equation*}
F(s)=e^{-s}\frac{s}{1+s^{2}}
\end{equation*}

Taking the inverse Laplace transform and get 
\begin{equation*}
f\left( x\right) =\cos (x-1)u(x-1)
\end{equation*}

to be the required solution defined on the half line $[0,\infty )$.
\end{proof}

\ \ \ \ \ \ \ \ \ \ \ \ \ \ \ \ \ \ \ \ \ \ \ \

\end{document}